\documentclass[10pt]{amsart}

\usepackage[ansinew]{inputenc}
\usepackage[psamsfonts]{amssymb}
\usepackage[all]{xy}

\numberwithin{equation}{section}

\theoremstyle{plain} 
\newtheorem{theor}[equation]{Theorem}
\newtheorem{cor}[equation]{Corollary}
\newtheorem{lem}[equation]{Lemma}
\newtheorem{proposition}[equation]{Proposition}

\theoremstyle{definition}

\newtheorem{defin}[equation]{Definition}

\theoremstyle{remark}

\newtheorem{rem}[equation]{Remark}
\newtheorem{ex}[equation]{Example}

\def\hur{\mathit{hur}}

\newcommand{\J}{\mathbb{J}}

\def\build#1_#2^#3{\mathrel{\mathop{\kern0pt#1}\limits_{#2}^{#3}}}

\begin{document}
\title{Batalin-Vilkovisky algebras and the $J$-homomorphism}
\author{Gerald Gaudens* and Luc Menichi**$^\dagger$}
\address{Gerald Gaudens\\Mathematisches Institut\\Beringstrasse 1\\
D-53115 Bonn, GERMANY\\\newline\indent Luc Menichi\\UMR 6093 associ\'ee au CNRS\\
Universit\'e d'Angers, Facult\'e des Sciences\\
2 Boulevard Lavoisier\\49045 Angers, FRANCE}
\email{geraldgaudens at math.uni-bonn.de, luc.menichi at univ-angers.fr}
\begin{abstract}
Let $X$ be a topological space.
The homology of the iterated loop space $H_*\Omega^n X$
is an algebra over the homology of the framed $n$-disks operad
$H_*f\mathcal{D}_n$ ~\cite{Getzler:BVAlg,Salvatore-Wahl:FrameddoBVa}.
We determine completely this $H_*f\mathcal{D}_n$-algebra structure on
 $H_*(\Omega^n X;\mathbb{Q})$.
We show that the action of $H_*(SO(n))$ on the iterated loop
space $H_*\Omega^n X$ is related to the $J$-homomorphism
and that the $BV$-operator vanishes on spherical classes only
in characteristic other than $2$.
\end{abstract}
\thanks{*University of Bonn, Germany}
\thanks{**University of Angers, France}
\thanks{$^\dagger$The second author was partially supported by the Mathematics Research Center of Stanford University}

\maketitle

\section{Introduction}
\label{intro}
Except when specified, the homology functor $H_*$ is considered
over an arbitrary field $\mathbb{K}$.
Let $n\geq 2$.
Let $X$ be a pointed topological space.
Consider the iterated loop space $\Omega^n X$.
The action of the little $n$-disks operad $\mathcal{D}_n$ on
$\Omega^n X$ gives $H_*\Omega^n X$ the structure of
$H_*\mathcal{D}_n$-algebra.
Fred Cohen~\cite{Cohen-Lada-May:homiterloopspaces} has showed that this gives $H_*\Omega^n X$ the structure
of an $e_n$-algebra.

\smallskip\begin{defin}
\label{definition algebre de Gerstenhaber}
A {\it $e_n$-algebra} is a
commutative graded algebra $A$
equipped with a linear map
$\{-,-\}:G \otimes G \to G$ of degree $n-1$
such that:

\noindent a) the bracket $\{-,-\}$ gives $A$ the structure of graded
Lie algebra of degree $n-1$. This means that for each $a$, $b$ and $c\in A$

$\{a,b\}=-(-1)^{(\vert a\vert+n-1)(\vert b\vert+n-1)}\{b,a\}$ and 

$\{a,\{b,c\}\}=\{\{a,b\},c\}+(-1)^{(\vert a\vert+n-1)(\vert b\vert+n-1)}
\{b,\{a,c\}\}.$

\noindent b)  the product and the Lie bracket satisfy the Poisson
relation:
$$\{a,bc\}=\{a,b\}c+(-1)^{(\vert a\vert+n-1)\vert b\vert}b\{a,c\}.$$
\end{defin}

Suppose that $X$ is equipped with a pointed action of $SO(n)$:
pointed action means that for any element $g$ in $SO(n)$, we have $g.*=*$.
Getzler~\cite{Getzler:BVAlg}, Salvatore and Wahl~\cite{Salvatore-Wahl:FrameddoBVa}
have noticed that the little $n$-disks operad $\mathcal{D}_n$
is a $SO(n)$-operad which acts in the category of $SO(n)$-spaces
on $\Omega^n X$, or equivalently that the framed $n$-disks operad
$f\mathcal{D}_n$ acts on $\Omega^n X$.
Therefore $H_*\Omega^n X$ is a $H_*SO(n)$-algebra
over the $H_*SO(n)$-operad $H_*\mathcal{D}_n$ or equivalently
is an algebra over the operad $H_*f\mathcal{D}_n$.
The first goal of this paper is to provide some
explicit computations of this structure
on $H_*\Omega^n X$ for various $X$ and various coefficients $\mathbb{K}$.

It is obvious that the structure of $H_*f\mathcal{D}_n$-algebra
is the structure of an $e_n$-algebra together with the structure of
$H_*SO(n)$-module which satisfy some compatibility relations between them.
Over any coefficients $\mathbb{K}$ when $n=2$, Getzler~\cite{Getzler:BVAlg}
has shown that a $H_*(f\mathcal{D}_2;\mathbb{K})$-algebra is a
$BV_2$-algebra (i.e. Batalin-Vilkovisky algebra).

\smallskip\begin{defin}~\cite[Def 5.2]{Salvatore-Wahl:FrameddoBVa}\label{definition BV_n-algebre}
A {\it $BV_n$-algebra $A$} is an $e_n$-algebra with a linear endomorphism
$BV:A\rightarrow A$ of degree $n-1$ such that $BV\circ BV=0$ and
for each $a,b\in A$,
\begin{equation}\label{crochet defaut de BV derivation}
\{a,b\}=(-1)^{\vert a\vert}\left(BV(ab)-(BVa)b-(-1)^{\vert a\vert}
a(BVb)\right).
\end{equation}
The bracket measures the deviation of the operator $BV$ from being a
derivation
with respect to the product.
Furthermore, in a $BV_n$-algebra, $BV$ satisfies the formula
~\cite[Proposition 1.2]{Getzler:BVAlg}
or ~\cite[(7) in section 5]{Salvatore-Wahl:FrameddoBVa}
\begin{equation}\label{BV derivation pour le crochet}
BV [a,b]=[BV a,b]+(-1)^{\vert a+1\vert}[a,BV b].
\end{equation}

\end{defin}
When $\mathbb{K}=\mathbb{Q}$ and for any $n\geq 2$, Salvatore and Wahl have
computed exactly
what $H_*(f\mathcal{D}_n;\mathbb{Q})$-algebras are:

\smallskip\begin{theor}~\cite[Th. 5.4]{Salvatore-Wahl:FrameddoBVa}
\label{Structure de fdn-algebres rationnellement}
A $H_*(f\mathcal{D}_n;\mathbb{Q})$-algebra $A$ is an $e_n$-algebra
such that the generators of the algebra $H_*(SO(n);\mathbb{Q})$
acts on $A$ as derivation with respect to the product,
except when $n$ is even; in the latter case the generator in degree $n-1$ defines
a structure of $BV_n$-algebra on $A$.
\end{theor}

We recall that $H_*(SO(2k+1);\mathbb{Q})$ is an exterior algebra  in generators $a_{4i-1}$ for $1\leq i \leq k$ while
$H_*(SO(2k+2);\mathbb{Q})$ is isomorphic to $H_*(SO(2k+1);\mathbb{Q})$ with an adjoined exterior generator $a_{2k+1}$ \cite[p. 300]{hatch}.

Our first theorem (Theorem~\ref{structure rationnelle des lacets iteres}) is that for many pointed $SO(n)$-spaces $X$,
this $H_*(f\mathcal{D}_n;\mathbb{Q})$-algebra structure on
$H_*(\Omega^n X;\mathbb{Q})$ reduces to:
\begin{itemize}
\item an $e_n$-algebra structure if $n$ is odd,
\item a $BV_n$-algebra structure if $n$ is even.
\end{itemize}
That is, most of the generators of
$H_*(SO(n);\mathbb{Q})$
act trivially in cases of interest.
Theorem~\ref{structure rationnelle des lacets iteres} holds in particular for any pointed topological space $X$
considered as a trivial pointed $SO(n)$-space.
In the rest of this paper, this is the only case that we will be dealing with.
We explain why in section~\ref{decomposition action diagonale}.

In section \ref{compuitation:ration}, we compute
the $BV_n$-algebra structure on
$H_*(\Omega^n X;\mathbb{Q})$ when $X$ is $n$-connected:
we show that the $BV_n$-algebra $H_*(\Omega^n X;\mathbb{Q})$
is a free  $BV_n$-algebra
(Theorem~\ref{Calcul BV_n algebre rationnellement}).


In section \ref{actionson}, we explain (Theorem~\ref{relation J-homomorphism et action de SO(n)}) how the action of
$H_* SO(n)$ on $H_*\Omega^n X$ is related to the $J$-homomorphism, a useful fact used later in our computations.
In section~\ref{compuitation:ration}
(Theorem~\ref{Calcul BV_n algebre rationnellement})
over $\mathbb{Q}$ and then more generally
in section~\ref{actionson} (Theorem~\ref{relation J-homomorphism et action de SO(n)}) over any field $\mathbb{K}$ of characteristic
different from $2$, we see that the $BV$ operator
vanishes on spherical classes.

In section~\ref{sec:bvspheres} , we show (Theorem~\ref{bvcalc:omegasdeux})
that the $BV$ operator
$$BV:H_1(\Omega^2S^3 , \mathbb{F}_2)\buildrel{\cong}\over\rightarrow 
H_2(\Omega^2S^3 , \mathbb{F}_2)$$
is an isomorphism.
Therefore over $\mathbb{F}_2$, the $BV$ operator is in general non trivial
on spherical classes.

Finally we notice that

-Theorem~\ref{bvcalc:omegasdeux} is a crucial step
in the main result of~\cite{menichi:stringtopspheres} and

-Kallel and Salvatore~\cite[Proposition 7.46]{Kallel:livre}
have given us an independant proof of this Theorem.

\smallskip {\it Acknowledgements.-}
The second author would like to thanks Ralph Cohen for
a discussion simplifying the proof of
Theorem~\ref{relation J-homomorphism et action de SO(n)}.

\section{$H_*(f\mathcal{D}_n;\mathbb{Q})$-algebra structures on
$H_*(\Omega^n X;\mathbb{Q})$}
\label{HQ:algebra:structures}
Denote by $\Omega_k^n S^n$  the path-connected component of $\Omega^n S^n$
given by pointed
maps of degree $k$.
For $n=\infty$, we denote the colimit of the spaces $\Omega^n S^n$ under the suspension maps by $QS^0$, which is the infinite loop space associated to the stable homotopy groups of spheres, that is $\pi_* QS^0$ is the ring  $\pi_*^S$ of stable homotopy. In this case, the degree $i$ component of $QS^0$ is denoted by $Q_iS^0$. In all cases, $[k]$ denotes the class of a degree $k$ map in  $\pi_0 \Omega^n S^n \cong \mathbb{Z}$.
The special orthogonal group $SO(n)$ acts on the sphere $S^{n-1}$
non pointedly.
By considering the sphere $S^n$ as the non reduced suspension of $S^{n-1}$,
we obtain a pointed $SO(n)$-action on $S^n$.
For example, by rotating, the `earth' $S^2$ has an action of the circle $S^1$ preserving the North pole.
By adjunction, we have a morphism of monoids
$\Theta:(SO(n),1)\rightarrow (\Omega_1^n S^n,id_{S^n})$.
Recall that $\Omega_1^n S^n$ is the monoid of self-homotopy
equivalences homotopic to the identity.

\smallskip\begin{theor}\label{structure rationnelle des lacets iteres}
Let $X$ be a pointed $SO(n)$-space.
Suppose that the action of $SO(n)$ on $X$
is obtained by
restriction along the morphism
$\Theta:SO(n)\rightarrow \Omega_1^n S^n$ from some action of $\Omega_1^n S^n$ on $X$.
Then the action
of $H_i(SO(n);\mathbb{Q})$ on $H_*(\Omega^n X;\mathbb{Q})$
coming from the diagonal action of $SO(n)$ on $\Omega^n X$
is trivial except maybe if $n$ is even and $i=n-1$.
\end{theor}

The case where $n$ is even and $i=n-1$  is analyzed in section \ref{compuitation:ration}.
Here are the most important examples of pointed $SO(n)$-spaces we can apply
Theorem~\ref{structure rationnelle des lacets iteres} to.
\begin{ex}\label{suspension exemple}
Let $Y$ be a pointed space. We consider the  $n^{\mathrm{th}}$ reduced suspension of $Y$, which is by definition
$X=S^n\wedge Y$. The space $X$ has a pointed $SO(n)$-action defined by
$
g.(t\wedge y):=(g.t)\wedge y
$
for any $g\in SO(n)$, $t\in S^n$ and $y\in Y$.
\end{ex}
\begin{ex}\label{trivial SO(n)-space}
Any pointed topological space $X$ can be considered as a trivial pointed
$SO(n)$-space.
\end{ex}
We give immediately the proof of Theorem~\ref{structure rationnelle des
lacets iteres}
since it will gives us the opportunity to review in general the diagonal
action of $SO(n)$ on $\Omega^n X$, for a pointed $SO(n)$-space $X$.
\begin{proof}
Let $G$ be a H-group acting pointedly on two spaces $Y$ and $Z$.
The diagonal action on the space of pointed maps, $map_*(Y,Z)$, is the
action
defined by
$$
(g.f)(y):=g.f(g^{-1}.y)
$$
for any $g\in G$, $y\in Y$ and any pointed map $f:(Y,*)\rightarrow (Z,*)$.
The diagonal action is natural up to homotopy with respect to morphisms
of H-groups.
Let $G:=SO(n)$, $Y:=S^n$ and $Z$ be any pointed $SO(n)$-space $X$, we obtain
the action considered by Geztler, Salvatore and
Wahl~\cite[Exemple 2.5]{Salvatore-Wahl:FrameddoBVa}.

The monoid $(\Omega_1^n S^n,id)$ is path-connected. So it is a
H-group~\cite[X.2.2]{Whitehead:eltsoht}
and $\Theta: (SO(n),1) \rightarrow (\Omega_1^n S^n,id)$
is a morphism of H-groups.

Now suppose that the pointed action of $SO(n)$ on $X$ is obtained by
restriction of an action of $\Omega_1^n S^n$.
Then by naturality, the diagonal action of $SO(n)$ on $\Omega^n X$
is homotopic to the composite
$$
SO(n)\times \Omega^n X\buildrel{\Theta\times \Omega^n X}\over\longrightarrow
\Omega^n_1 S^n\times\Omega^n X\buildrel{action\times\Omega^n
X}\over\longrightarrow\Omega^n X.
$$
where action is the diagonal action of $\Omega_1^n S^n$ on $\Omega^n X$.

As we will see in more details below (beginning of section~\ref{actionson}), there is a pointed homotopy
equivalence
between $(\Omega_0^n S^n,*)$ and $(\Omega_1^n S^n,id)$.
But
{\Small
$$\forall i\geq 1,\quad
\pi_i(\Omega_0^n S^n,*)\otimes\mathbb{Q}\cong
\pi_{i+n}(S^n,*)\otimes\mathbb{Q}=
\begin{cases}
\mathbb{Q} &\mbox{if $n$ is even and
$i=n-1$,}\\
0& \mbox{otherwise.}
\end{cases}
$$}
Therefore for $n$ even, $\Omega_1^n S^n$ is rationally homotopy equivalent
to $S^{n-1}$.
And for $n$ odd, $\Omega_1^n S^n$ is rationally contractible.
\end{proof}

\section{Decomposition of the action of $SO(n)$ on $\Omega^n X$ for a general $SO(n)$-space}
\label{decomposition action diagonale}
The diagonal action of $SO(n)$
on $\Omega^n X$ is the combination of two different actions of
$\Omega^n X$:
\begin{enumerate}
\item The action denoted by $S$, `on the source', given by
$
(g.f)(y):=f(g^{-1}.y)
$
for any $g\in SO(n)$, $y\in Y$ and any pointed map $f:(Y,*)\rightarrow (Z,*)$.
\item The action denoted by $T$, `on the target', given by
$
(g.f)(y):=g.f(y)
$
for any $g\in SO(n)$, $y\in Y$ and any pointed map $f:(Y,*)\rightarrow (Z,*)$.
\end{enumerate}
For $n=2$, the result in this section is:

\smallskip\begin{theor}
\label{theo7}
Let $X$ be a pointed $S^1$-space.
Denote by $BV_{diag}$, $BV_S$ and $BV_T$ the $BV$-operators respectively due to the diagonal action, the action `on the
source' and the action `on the target' of $S^1$
on $\Omega^2 X$.
Then $BV_{diag}=BV_S+BV_T$
and $BV_T$ is a derivation with respect to the Pontryagin product.
\end{theor}
The diagonal action is in homology the composite (where the roles of $S$ and
$T$
could be permuted)
$$
\xymatrix{
H_*SO(n)\otimes H_*\Omega^n X
\ar[d]^{\Delta_{H_*SO(n)}\otimes H_*\Omega^n X}\\
H_*SO(n)\otimes H_*SO(n)\otimes H_*\Omega^n X
\ar[d]^{H_*S\otimes H_*\Omega^n X}\\
H_*SO(n)\otimes H_*\Omega^n X
\ar[d]^{H_*T}\\
H_*\Omega^n X.
}
$$
For example when $n=2$, let $[S^1]$ be the fundamental class of $H_*SO(2)$.
Since $[S^1]$ is a primitive element,
the $BV$-operator $BV_{diag}$ on
$H_*\Omega^2 X$, which is due to the diagonal action
in homology of $[S^1]$, is the sum
of the two operators $BV_{S}$ and $BV_{T}$ on
$H_*\Omega^2 X$ given by the action $S$ on the source and the action $T$
on the target of $[S^1]$.

Let $\circ:map_*(X,Y)\times \Omega^n X\rightarrow \Omega^n Y$, $(g,f)\mapsto
g\circ f$
be the composition map.
We have the following obvious distributive law between composition $\circ$ and loop
multiplication due the structure of co-H-group on $S^n$.
For any $f\in map_*(X,Y)$, $g,h\in\Omega^n X$,
\begin{eqnarray}
\label{distributive law}
f\circ(gh)=(f\circ g)(f\circ h).
\end{eqnarray}
In the particular case $Y=X$, this means that the multiplication of loops
$$
\Omega^n X\times\Omega^n X\rightarrow \Omega^n X
$$
is $map_*(X,X)$-equivariant.
Since $X$ is a pointed $SO(n)$-space, we have a morphism of monoid
$SO(n)\rightarrow map_*(X,X)$.
Therefore the multiplication of loops
$$
\Omega^n X\times\Omega^n X\rightarrow \Omega^n X
$$
is $SO(n)$-equivariant with respect to the action `on the target'.
So in homology, the Pontryagin product
$$
H_*\Omega^n X\times H_*\Omega^n X\rightarrow H_*\Omega^n X
$$
is a morphism of $H_*SO(n)$-modules for the action `on the target'.

For example, when $n=2$, the action  of $[S^1]$ in homology `on the target'
gives an operator $BV_T$ which is a derivation with respect to the Pontryagin
product
and so does not contribute at all to the bracket (See formula~(\ref{crochet defaut de BV derivation})).
Therefore we do not find the action `on the target' of $SO(n)$ on $\Omega^n
S^n$
interesting.
In the rest of the paper,
we will study only the action on the source of $SO(n)$ on $\Omega^n X$ or
equivalently
we will considered the diagonal action of $SO(n)$ on $\Omega^n X$ where
$X$ is a trivial pointed $SO(n)$-space (Example~\ref{trivial SO(n)-space}).

\section{Computation of the $H_*(f\mathcal{D}_n;\mathbb{Q})$-algebra $H_*(\Omega^n X;\mathbb{Q})$}
\label{compuitation:ration}
\noindent{\bf Free $e_n$-algebras.} The suspension of a graded vector space $V$ is the graded vector space $sV$
such that $(sV)_{i+1}=V_i$.
Let $L$ be a graded Lie algebra.
The $(n-1)$ desuspension of $L$, $s^{1-n}L$, has a Lie bracket of degree $n-1$.
The free graded commutative algebra $\Lambda(s^{1-n}L)$ is equipped with
a unique structure of $e_n$-algebra such that the inclusion
$s^{1-n}L\hookrightarrow\Lambda(s^{1-n}L)$ commutes with the brackets.
This inclusion is universal.
When $n=1$, this bracket is the well-known Schouten bracket of the Poisson
algebra $\Lambda L$.
Let $A$ be an $e_n$-algebra. Then any Lie algebra morphism $s^{1-n}L \rightarrow A$
extends to an unique morphism of $e_n$-algebras
$\Lambda(s^{1-n}L)\rightarrow A$.
Let $L:=\pi_*(\Omega X)\otimes\mathbb{Q}$ equipped with the Samelson bracket
and $A:=H_*(\Omega^n X;\mathbb{Q})$.
Then $s^{1-n}L=\pi_*\Omega^n X\otimes\mathbb{Q}$ and we have:

\smallskip\begin{theor}\cite[Remark 1.2
  p. 214]{Cohen-Lada-May:homiterloopspaces}\cite{CohenF:configshLiea}\label{Calcul algebre de Gerstenhaber rationnellement}
Let $X$ be an $n$-connected topological space.
The Hurewicz morphism induces an isomorphism
$$
\Lambda(\pi_*\Omega^n X\otimes\mathbb{Q})\buildrel{\cong}\over\rightarrow H_*(\Omega^n X;\mathbb{Q})
$$
of both $e_n$-algebras
and Hopf algebras between the free $e_n$-algebra and the rational
homology of the iterated loop space.
\end{theor}
\noindent{\bf Free $BV_n$-algebras.}
Suppose now that $n\geq 0$ and is even and that the graded Lie
algebra $L$ is equipped with a differential $d_L$.
Let $d_0$ be the derivation of degree $n-1$ on $\Lambda(s^{1-n}L)$
given by
{\small
$$
d_0(s^{1-n}x_1\wedge\dots\wedge s^{1-n}x_k)=-\sum_{i=1}^k
(-1)^{n_i} s^{1-n}x_1\wedge\dots\wedge s^{1-n}d_L x_i\wedge\dots\wedge s^{1-n}x_k 
$$
}
where $ n_i=\sum_{j<i}\vert s^{1-n} x_j\vert$. Since $d_L^2=0$,
$d_0^2=0$.
Let $d_1$ be the endomorphism of degree $n-1$ on $\Lambda(s^{1-n}L)$
given by
{\small
\begin{multline*}
d_1(s^{1-n}x_1\wedge\dots\wedge s^{1-n}x_k)=\\
\sum_{1\leq i<j\leq k}
(-1)^{\vert x_i-n+1\vert}(-1)^{n_{ij}}
s^{1-n}\{x_i,x_j\}\wedge s^{1-n}x_1\cdots \widehat{s^{1-n} x_i}\cdots
\widehat{s^{1-n}x_j}\cdots \wedge s^{1-n}x_k~~.
\end{multline*}
}
The symbol $\hat{\quad}$ means `deleted'.
Here the sign $(-1)^{n_{ij}}$ is such that
$s^{1-n}x_1\wedge\dots\wedge s^{1-n}x_k=
(-1)^{n_{ij}}
s^{1-n}x_i\wedge s^{1-n}x_j\wedge s^{1-n}x_1\wedge\dots \widehat{s^{1-n} x_i}\dots
\widehat{s^{1-n}x_j}\dots \wedge s^{1-n}x_k$.
Since $d_L$ is a derivation with respect to the bracket, we have
$d_0d_1+d_1d_0=0$.
The Jacobi identity implies that $d_1^2=0$.

Consider the endomorphism $BV$ of degree $n-1$ on $\Lambda(s^{1-n}L)$
defined by $BV:=d_0+d_1$. We have $BV\circ BV=0$.
A direct calculation shows that
$$
(-1)^{\vert a\vert}\left(BV(ab)-(BVa)b-(-1)^{\vert a\vert}
a(BVb)\right)
$$ is the bracket on the $e_n$-algebra $\Lambda(s^{1-n}L)$.
Therefore the $e_n$-algebra $\Lambda(s^{1-n}L)$ equipped with this
linear operator $BV$ is a $BV_n$-algebra, that we will denote
$\Lambda s^{1-n}(L,d_L)$ and
that we will call the {\it free $BV_n$-algebra
on the differential graded Lie algebra $(L,d_L)$}.
The inclusion
$s^{1-n}(L,d_L)\hookrightarrow\Lambda s^{1-n}(L,d_L)$ commutes with the
brackets
and the differentials $-d_L$ and $BV$.
Again, this inclusion is universal.
Let $A$ be an $BV_n$-algebra.
Then any differential graded Lie algebra morphism $s^{1-n}(L,d_L)\rightarrow A$
extends to a unique morphism of $BV_n$-algebras
$\Lambda s^{1-n}(L,d_L)\rightarrow A$.

Suppose moreover that $\mathbb{K}:=\mathbb{Q}$.
Then the differential $d_1$ is the unique coderivation on
$\Lambda(s^{1-n}L)$
decreasing wordlength by $1$
such that
\begin{equation}\label{differentielle de Cartan-Chevalley-Eilenberg}
d_1(s^{1-n}x\wedge s^{1-n}y)=(-1)^{\vert x \vert -n+1} s^{1-n}\{x,y\}.
\end{equation}
When $n=0$, our operator $BV$ on $\Lambda(s^{1-n}L)$
coincide with the differential of the Cartan-Chevalley-Eilenberg complex~\cite[p. 301]{Felix-Halperin-Thomas:ratht}
whose homology is $\mbox{Tor}_*^{U(L,d_L)}(\mathbb{Q},\mathbb{Q})$.
The case $n=2$ is considered in~\cite[Section 1.1 p. 10]{Tamarkin-Tsygan:ncdchBValg}.
We prove:

\smallskip\begin{theor}\label{Calcul BV_n algebre rationnellement}
Suppose that $n\geq 2$ and is even. Let $X$ be a $n$-connected topological space.
The $BV_n$-algebra $H_*(\Omega^n X;\mathbb{Q})$ is isomorphic to
$\Lambda s^{1-n}(\pi_*\Omega X\otimes\mathbb{Q},0)$
the free $BV_n$-algebra on the graded Lie algebra
$\pi_*\Omega X\otimes Q$ equipped with the zero differential.
\end{theor}
\begin{proof}
By Theorem~\ref{Calcul algebre de Gerstenhaber rationnellement}, the
underlying $e_n$-algebras are isomorphic.
By the universal property of the free $BV_n$-algebra
$\Lambda s^{1-n}(L,0)$, its suffices to show that
the BV-operator on $H_*(\Omega^n X;\mathbb{Q})$ vanishes on spherical
classes.

The $BV$-operator on $H_*(\Omega^n X;\mathbb{Q})$ is induced by the action of a primitive element of degree $n-1$
in $H_*(SO(n);\mathbb{Q})$.
More generally, let $Y$ be a $SO(n)$-space.
The operator induced by the action of a primitive element in $H_*SO(n)$ is a coderivation (coaugmented if the action
is pointed). The image of a primitive element by a coaugmented coderivation is a primitive element.
Therefore, the operator $BV$ induces an operator of degree $n-1$ on the primitive
elements of $H_*(\Omega^n X;\mathbb{Q})$, denoted by $PH_*(\Omega^n X;\mathbb{Q})$.
Let $i\geq 1$. By Cartan-Serre's theorem,
$PH_i(\Omega^n X;\mathbb{Q})\cong \pi_i\Omega^n X\otimes\mathbb{Q}\cong\pi_{i+n}X\otimes\mathbb{Q}$.
Denote by $X_\mathbb{Q}$ the rationalization of $X$.
By the universal property of localization, $\pi_{i+n}X\otimes\mathbb{Q}\cong \pi_{i+n}X_\mathbb{Q}
\cong [S^{i+n}_\mathbb{Q},X_\mathbb{Q}]$.
So the operator $BV$ can be identified with a morphism between the pointed homotopy classes
$$ [S^{i+n}_\mathbb{Q},X_\mathbb{Q}]\rightarrow [S^{i+2n-1}_\mathbb{Q},X_\mathbb{Q}].$$

Now since $X$ was equipped with the structure of trivial $SO(n)$-pointed space,
the $BV$ operator is natural with respect to continuous maps.
This $BV$ operator can be therefore identified with a `rational homotopy operation':
a natural transformation $ [S^{i+n}_\mathbb{Q},X_\mathbb{Q}]\rightarrow [S^{i+2n-1}_\mathbb{Q},X_\mathbb{Q}]$.
By~\cite[XI.1.2]{Whitehead:eltsoht}, this rational homotopy operation is
the composition by an element of $\pi_{i+2n-1}S^{i+n}\otimes\mathbb{Q}$.
Since for $i\geq 1$, $\pi_{i+2n-1}S^{i+n}\otimes\mathbb{Q}$ is trivial, the $BV$ operator
is null on $\pi_i\Omega^n X\otimes\mathbb{Q}$ for $i\geq 1$.
\end{proof}
Remark that in fact, in this proof, we have that the BV-operator on
$H_*(\Omega^n X;\mathbb{Q})$ is the unique coderivation decreasing
wordlength by $1$
satisfying~(\ref{differentielle de Cartan-Chevalley-Eilenberg}).
Therefore we have recovered without computations that
$\Lambda (s^{-(n-1)}L)$ is a $BV_n$-algebra
in the case $L:=\pi_*\Omega X\otimes\mathbb{Q}$.
This was our starting observation.

\begin{rem}
Suppose more generally that $X$ is a $SO(n)$-pointed space
not necessarly trivial.
The generator in degree $n-1$ which defines the structure of $BV_n$-algebra
(Theorem~\ref{Structure de fdn-algebres rationnellement})
on $H_*(\Omega^n X;\mathbb{Q})$ is also primitive.
Therefore Theorem~\ref{theo7} holds also for the operator $BV=BV_{diag}$
of degree $n-1$ on $H_*(\Omega^n X;\mathbb{Q})$.
In particular, $BV=BV_S+BV_T$.
Moreover, this operator $BV$ induces a differential $d_L$
on $\pi_*\Omega X\otimes\mathbb{Q}$.
By~(\ref{BV derivation pour le crochet})
and~\cite[Remark 1.2 p. 214]{Cohen-Lada-May:homiterloopspaces},
 $d_L$ is a derivation with respect
to the Samelson bracket. And this time,
the $BV_n$-algebra $H_*(\Omega^n X;\mathbb{Q})$ is isomorphic to
$\Lambda s^{1-n}(\pi_*\Omega X\otimes\mathbb{Q},d_L)$
the free $BV_n$-algebra with the (in general non-zero) differential $d_L$.
By Theorem~\ref{Calcul BV_n algebre rationnellement}, the differential $d_1$ corresponds to $BV_S$ while
$d_0$ corresponds to $BV_T$.
\end{rem}
Finally, remark that if $X=M$ is a manifold and $n=2$,
Salvatore and Wahl~\cite[Theorem 6.5]{Salvatore-Wahl:FrameddoBVa}
have also given a formula for the $BV_2$-algebra 
$H_*(\Omega^2 M;\mathbb{Q})$.

\smallskip In the following section, we show that over any field $\mathbb{K}$, the action 
of a spherical class in $H_*SO(n)$ on a spherical class in  $H_*\Omega^n X$
still corresponds to a homotopy operation. But this homotopy operation
is generally not trivial and is related to the well-known $J$-homomorphism.

\section{Action of $SO(n)$ on iterated loop spaces and $J$-homomorphism}
\label{actionson}

Let $n\geq 2$. The case of infinite loop spaces is implicitly included as the case $n= \infty$.
In this case, $SO (n)$ is meant to be the infinite orthogonal group $SO$. 
The iterated loop space $(\Omega^n S^n,*)$ is a H-group.
So the multiplication by $id$ gives a free homotopy equivalence
$$\mbox{multiplication by id}:\Omega_0^n S^n\stackrel{\simeq}{\longrightarrow}
\Omega_1^n S^n.$$
Since both $(\Omega_0^n S^n,*)$ and $(\Omega_1^n S^n,id)$ are
H-spaces, a map homotopic to the multiplication by $id$,
will give a pointed homotopy equivalence from $(\Omega_0^n S^n,*)$ to
$(\Omega_1^n S^n,id)$~\cite[III.1.11 and III.4.18]{Whitehead:eltsoht}.
The pointed homotopy inverse of this map is a map homotopic to the
multiplication by a degree $-1$ map
from $S^n$ to $S^n$.
Let $inv:SO(n)\rightarrow SO(n)$ be the map sending an orthogonal matrix to its inverse.
Recall that we have a morphism of topological monoids
$$
\Theta: (SO(n),1) \rightarrow (\Omega_1^n S^n,id)~~.
$$
Denote by
$ad_n:\pi_{i+n}(X)\buildrel{\cong}\over\rightarrow\pi_{i}(\Omega^n X)$
the adjunction map.
The classical $J$-homomorphism that we will denote $\J$, is the composite 
$$\pi _i(SO(n),1) \buildrel{\pi_i(J)}\over\rightarrow \pi_i(\Omega^n S^n,*)
\buildrel{ad_n^{-1}}\over\rightarrow \pi_{i+n}(S^n,*)$$
 where $J$ is the composite
\begin{equation}\label{decomposition du J-homomorphisme}
J:(SO(n),1) \buildrel{\Theta}\over\rightarrow (\Omega_1^n
S^n,id)\buildrel{\simeq}\over\leftarrow (\Omega_0^n S^n,*)~~.
\end{equation}
Let $J'$ be the map
$$
J':(SO(n),1)\buildrel{inv}\over\rightarrow(SO(n),1)  \buildrel{\Theta}\over\rightarrow (\Omega_1^n
S^n,id)\buildrel{\simeq}\over\leftarrow (\Omega_0^n S^n,*)
$$
That is $J'$ is the precomposition of $J$ with the inverse map of $SO(n)$. The following lemma is a classical one in algebraic topology:

\smallskip\begin{lem}
\label{j:versus:j'}
In homotopy, $\pi_*J'=-\pi_*J$ while in homology $J_*'= J_* \chi$ where $\chi$ is
the antipode of the Hopf algebra $H_*SO(n)$. In particular, on a primitive element $a\in H_* SO (n)$,
$J'_*(a)=-J_*(a)$. 
\end{lem}
Let $\circ:(\Omega^n X,*)\times (\Omega^n S^n,*)\rightarrow (\Omega^n X,*)$, $(g,f)\mapsto
g\circ f$
be the composition map.
Using the distributive law~(\ref{distributive law}),
we have for any $g\in \Omega^n X$, $f\in\Omega^n S^n$:
$$
g\circ(f.\mbox{id}_{S^n})=(g\circ f)(g\circ \mbox{id}_{S^n})=(g\circ f)g.
$$
The left action of $SO(n)$ on $\Omega^n X$ is given by the diagram commutative up to
free homotopy
$$
\xymatrix{
SO(n)\times \Omega^n X\ar[r]^{J'\times\Omega^n X}\ar[d]_{(\Theta\circ inv)\times\Omega^n X}
&\Omega^n_0 S^n\times \Omega^n X\ar[dl]_{\mbox{mult by id}}\ar[d]^{\Omega^n
S^n\times\Delta_{\Omega^n X}}\\
\Omega^n_1 S^n\times \Omega^n X\ar[d]_{\tau}
&\Omega^n_0 S^n\times \Omega^n X\times \Omega^n X\ar[d]^{\tau\times\Omega^n X}\\
\Omega^n X\times\Omega^n_1 S^n\ar[d]_{\circ}
&\Omega^n X\times \Omega^n_0 S^n\times\Omega^n X\ar[d]^{\circ\times\Omega^n X}\\
\Omega^n X^n
&\Omega^n X\times \Omega^n X\ar[l]^{mult}
}
$$
Here $\tau:\Omega^n S^n\times \Omega^n X\rightarrow \Omega^n
X\times\Omega^n S^n$
is the map that interchanges factors.
The top left triangle commute up to free homotopy, by definition of $J'$.
The bottom diagram commutes exactly.
Denote by $\circ:H_*\Omega^n X\otimes H_*\Omega^n S^n$, $(f\otimes
g)\mapsto f\circ g $,
the morphism induced in homology by composition
and by $J_*:H_*SO(n)\rightarrow H_*\Omega^n S^n$, the morphism induced
by $J$.
In homology, the action of $H_*SO(n)$ on $H_*\Omega^n
X$, denoted by $\cdot$,
is given for $f\in H_*SO(n)$ and $g\in H_*\Omega^n X$ by
\begin{eqnarray*}
f\cdot g=\sum (-1)^{\vert g'\vert\vert f\vert} \left(g' \circ
(J'_*f)\right)g''
\end{eqnarray*}
where $\Delta g=\sum g'\otimes g''$.
Note that alternatively to give a shorter proof of this equation,
we could have applied
\cite[Theorem 3.2 i) p. 363]{Cohen-Lada-May:homiterloopspaces}.
But we have preferred to give a independent simple proof.

The unit $1$ of the
algebra $H_*\Omega^n X$ is given by applying
homology
to the inclusion of the constant map $*\hookrightarrow\Omega^n S^n$.
Therefore by applying homology to the commuting diagram
$$
\xymatrix{
{*}\times\Omega^n S^n\ar[r]\ar[d]
& {*}\ar[d]\\
\Omega^n X\times\Omega^n S^n\ar[r]_\circ
&\Omega^n X
}
$$
we obtain~\cite[Proposition 3.7(i) p. 364]{Cohen-Lada-May:homiterloopspaces} that for any $f\in H_*\Omega^n S^n$,
$$1\circ f=\varepsilon (f)1.$$
Here $\varepsilon$ is the augmentation of the Hopf algebra $H_*\Omega^n
S^n$.
So, in homology, the action of $f\in H_{>0}SO(n)$ on a primitive element
$g\in H_*\Omega^n X$ is given by
\begin{eqnarray*}
f\cdot g=(-1)^{\vert f\vert\vert g\vert}g\circ (J'_*f)= (-1)^{\vert f\vert\vert g\vert +1}g\circ (J_*f) \quad .
\end{eqnarray*}
Otherwise said, the action of $H_{>0}SO(n)$ on primitive elements
of $H_*\Omega^n X$ is given in homology by the composite
{\Small
$$
(SO(n),1)\times (\Omega^n X,*)\buildrel{J'\times\Omega^n X}\over\rightarrow
(\Omega^n S^n,*)\times (\Omega^n X,*)
\buildrel{\tau}\over\rightarrow (\Omega^n X,*)\times(\Omega^n S^n,*)
\buildrel{\circ}\over\rightarrow
(\Omega^n X,*).
$$
}
Let $\bar\circ:\pi_i(\Omega^n X,*)\times \pi_j(\Omega^n
S^n,*)\rightarrow\pi_{i+j}(\Omega^n X,*)$
be the map induced on homotopy groups by the composition $\circ$.
Explicitly, the composition map
$$\circ:(\Omega^n X,*)\times (\Omega^n S^n,*)\rightarrow (\Omega^n X,*)$$
passing to the quotient, defines a map
$$(\Omega^n X,*)\wedge (\Omega^n S^n,*)\rightarrow (\Omega^n X,*).$$
The map $\bar\circ$ sends $(f,g)$ to the composite
$$f\bar\circ g:S^i\wedge S^j\buildrel{f\wedge g}\over\rightarrow
(\Omega^n X,*)\wedge (\Omega^n S^n,*)\rightarrow (\Omega^n X,*).$$

Let $X$ be an infinite loop space, the $\bar\circ$ product extends to a product
$$\bar\circ:(X,*)\times (QS^0,*)\rightarrow (X,*) \quad .$$
As before, the $\circ$ product induces a map
$$
\bar\circ: \pi_i X \times \pi_j QS^0 \longrightarrow \pi_{i+j} X \quad .
$$
Denote by $\hur $ the Hurewicz morphism.
Since the diagram
$$
\xymatrix{
S^i\times S^j\ar[r]^{f\times g}\ar[d]
& \Omega^n X\times\Omega^n S^n\ar[d]^{\circ}\\
S^i\wedge S^j\ar[r]_{f\bar\circ g}
& \Omega^n X
}
$$
\noindent is commutative, by applying homology we obtain that the
Hurewicz morphism commutes with the composition:
$$
\hur  (f)\circ \hur (g)=\hur (f\bar\circ g).
$$
This proof is similar to the proof that $\hur $ commutes with the Samelson
bracket~\cite[X.6.3]{Whitehead:eltsoht}.
So finally, we have proved:

\smallskip\begin{proposition}
\label{prop:prop:prop}
Let $f\in\pi_{i}SO(n)$, $g\in\pi_{j}\Omega^n X$, $i,j\geq 1$.
Then the action of $\hur  (f)\in H_iSO(n)$ on $\hur  (g)\in H_j\Omega^n X$,
is given by
\begin{eqnarray*}
\hur (f)\cdot \hur (g)=
-(-1)^{ij} \hur (g\bar\circ\pi_*J(f))=
-(-1)^{ij} \hur (g\bar\circ ad_n(\J f)).
\end{eqnarray*}
\end{proposition}

\noindent Recall that $ad_n:\pi_{i+n}X\buildrel{\cong}\over\rightarrow\pi_i\Omega^n X$
denotes the adjunction map. The following lemma follows from elementary
manipulation with the loop-suspension adjunction.

\smallskip\begin{lem}
\label{descrip:prod}
We have the following commutative diagram
$$
\xymatrix{
\pi_{i+n}X\times \pi_{j+n}S^n\ar[r]\ar[d]_{ad_n\times ad_n}^{\cong}
&\pi_{i+j+n}X\ar[d]^{ad_n}_{\cong}\\
\pi_{i}\Omega^n X\times \pi_{j}\Omega^n S^n\ar[r]_{\bar\circ}
&\pi_{i+j}\Omega^n X
}
$$
\noindent where the top arrow is the map sending $(f,g)$ to
$f\circ\Sigma^i g$, the composite of $f$ and the $i^ \mathrm{th}$ suspension of $g$.
\end{lem}
\noindent In the stable case, we simply obtain:
\begin{lem}
\label{lem:infinite:loop}
Let $X$ be an infinite loop space, and let $\widehat X$ be the corresponding spectrum. The 
$\bar\circ$ product coincides under the isomorphism $\pi_* X \cong \pi_*^S \widehat X$ with the natural action of the stable homotopy ring $\pi_*^S$ on the stable homotopy module $\pi_*^S\widehat X$.
\end{lem}
Recall that $\J:\pi_iSO(n)\rightarrow\pi_{i+n}(S^n)$
denote the classical $J$-homomorphism.

\smallskip\begin{theor}
\label{relation J-homomorphism et action de SO(n)}
Let $f\in\pi_{i}SO(n)$, $g\in\pi_{j+n} X$, $i,j\geq 1$.
Then the action of $\hur  (f)\in H_iSO(n)$ on
$(\hur \circ ad_n) (g)\in H_j\Omega^n X$, is given by
$$
\hur (f)\cdot (\hur \circ ad_n) (g)=-(-1)^{ij} (\hur \circ ad_n)(g\circ\Sigma^j\J f).
$$
\end{theor}
\begin{proof}
By applying Proposition \ref{prop:prop:prop} to $f$ and $ad_ng$,
and then Lemma~\ref{descrip:prod}
{\small $$
\hur (f)\cdot \hur (ad_ng)=-(-1)^{ij} \hur ((ad_ng)\bar\circ ad_n(\J f))=
-(-1)^{ij} \hur \circ ad_n(g\circ\Sigma^j \J f)~~.
$$}
\end{proof}
\begin{cor}\label{BV nulle en caracteristique differente de 2}
Let $X$ be a topological space. Let $g\in\pi_{j+2}(X)$, $j\geq 1$.
 
i) The BV operator $H_j(\Omega^2 X)\rightarrow H_{j+1}(\Omega^2 X)$
maps $(\hur \circ ad_2) (g)$ to $-(-1)^{j} (\hur \circ ad_2)(g\circ\Sigma^j\eta)$.

ii) In particular,
over a field $\mathbb{K}$ of characteristic different from $2$.
the $BV$ operator $H_j\Omega^2 X\rightarrow H_{j+1}\Omega^2 X$ is null
on spherical classes.
\end{cor}
\begin{proof}
Take $n=2$ and let $f:=id_{S^1}$ in
Theorem~\ref{relation J-homomorphism et action de SO(n)}.
The Hopf map $\eta:S^3\twoheadrightarrow S^2$ is equal to $\J (id_{S^1})$.
Its $j^{\mathrm{th}}$ suspension, $\Sigma^j\eta$ is the generator of
$\pi_{3+j}S^{2+j}\cong \mathbb{Z}/2\mathbb{Z}$.
Therefore localized away from $2$, $\Sigma^j\eta$ is null homotopic.
\end{proof}
Recall that an element in the homology of a space is called \emph{spherical}
if it sits in the image of the Hurewicz homomorphism.

\section{Batalin-Vilkovisky structure of $H_*(\Omega^2 S^3;\mathbb{F}_2)$}
\label{sec:bvspheres} 
In this section, we see that part ii) of
Corollary~\ref{BV nulle en caracteristique differente de 2}
is not true in characteristic $2$. Heretofore, homology will mean homology with coefficients in $\mathbb{F}_2$
and for any pointed space $X$, $\hur : \pi_* X \longrightarrow H_* (X, \mathbb{F}_2)$ will denote the modulo $2$
Hurewicz homomorphism. 
Recall the following classical splitting Lemma:
\smallskip\begin{lem}\label{splitting lemma}
Let $F\buildrel{j}\over\hookrightarrow
E\buildrel{p}\over\twoheadrightarrow B$
be a fibration up to homotopy with $F$, $E$ and $B$ path-connected.
Suppose that there is a map $s:B\rightarrow E$ such that the composite
$p\circ s$ is a homotopy equivalence and that $E$ is an H-space with
multiplication $\mu$.
Then the composite $\mu\circ (j\times s):F\times
B\buildrel{\simeq}\over\rightarrow E$ is a homotopy equivalence.
In particular $j$ admit a retract up to homotopy.
\end{lem}
The homology of $\Omega^2 {S}^{3}$, as a Pontryagin algebra, is
polynomial on generators $u_n$ of degree $2^n -1, n\geq 1$,
(see Cohen's work in~\cite{Cohen-Lada-May:homiterloopspaces}).
We first notice that $u_1$
is the bottom non trivial class in positive degrees, and as such, must
be in the image of the Hurewicz homomorphism, according to the
Hurewicz theorem. But $\pi_1 \Omega ^2 S^3$ is infinite cyclic
generated by
$\iota=\mathrm{ad}_2 (\mathrm{Id}_{S^3})$ the adjoint of the identity of $S^3$.

Since $u_1=\hur \iota=\hur \circ ad_2(id_{S^3})$, thanks to part i) of
Corollary~\ref{BV nulle en caracteristique differente de 2},
$$
{BV} (u_1) = \hur  (\mathrm{ad}_2(\Sigma\eta))
$$
where $\mathrm{ad}_2(\Sigma\eta) \in \pi_2 (\Omega^2 S^3)$ is the adjoint of suspension of the Hopf map $\Sigma\eta\in \pi_4 S^3$.

We want to show that the Hurewicz homomorphism for $\Omega^2 S^3$ is
non trivial in degree $2$. This can be seen as follows.
We know that $\pi _1 \Omega ^2 S^3 \cong\mathbb{Z}$ and
$\pi _2 \Omega^2 S^3 \cong\mathbb{Z}/2\mathbb{Z}$
generated by $ad_2(\Sigma \eta)$.

Let $p:S^3\rightarrow  K(\mathbb{Z},3)$ represent the generator
for the third integral cohomology group of $S^3$,
$\pi_3K(\mathbb{Z},3)\cong H^3(S^3;\mathbb{Z})$.
Let $S^3\langle 3\rangle$ be the homotopy fiber of $p$
and let $j:S^3\langle 3\rangle\rightarrow S^3$ be the fiber inclusion.

Let $\iota:S^1\rightarrow\Omega^2 S^3$ the adjoint of the identity of
$S^3$.
Since $S^3\langle 3\rangle$ is $3$-connected,
$\pi_1(\Omega^2 p)$ is an isomorphism and so maps $\iota$ to
$\pm id_{S^1}$.
Therefore by applying Lemma~\ref{splitting lemma} to the homotopy
fibration $\Omega^2 S^3\langle 3\rangle
\buildrel{\Omega^2 j}\over\rightarrow
\Omega^2 S^3\buildrel{\Omega^2 p}
\over\rightarrow \Omega^2 K(\mathbb{Z},3)\simeq S^1$,
we obtain that
$\Omega^2 j:\Omega^2 S^3\langle 3\rangle\rightarrow \Omega^2 S^3$
has a retract up to homotopy and so is injective in homology.
Since $\pi_3(S^1)=\pi_2(S^1)=0$,
 $\pi_2(\Omega^2 j):\pi_2(\Omega^2 S^3\langle 3\rangle)
\buildrel{\cong}\over\rightarrow
\pi_2(\Omega^2 S^3)\cong \mathbb{Z}/2\mathbb{Z}$
is an isomorphism.
Since $\Omega^2 S^3\langle 3\rangle$ is simply connected,
the Hurewicz homomorphism
$\hur :\pi_2(\Omega^2 S^3\langle 3\rangle)
\buildrel{\cong}\over\rightarrow
H_2(\Omega^2 S^3\langle 3\rangle)$
is an isomorphism.
Since $H_2(\Omega^2 j)$ is also an isomorphism,
the Hurewicz homomorphism
$$\hur :\pi_2(\Omega^2 S^3)={\mathbb{Z}}/{2\mathbb{Z}}\mathrm{ad}_2.\Sigma\eta
\buildrel{\cong}\over\longrightarrow
H_2(\Omega^2 S^3)={\mathbb{Z}}/{2\mathbb{Z}}.u_1^2$$
is an isomorphism. Hence
$\hur  (\mathrm{ad}_2(\Sigma\eta))=u_1^2$ and so
$BV(u_1)=u_1^2$.
So we have proved:
\smallskip\begin{theor}
\label{bvcalc:omegasdeux}
The Batalin-Vilkovisky operator
$BV:H_1(\Omega^2S^3 , \mathbb{F}_2)\buildrel{\cong}\over\rightarrow 
H_2(\Omega^2S^3 , \mathbb{F}_2)$
is non trivial.
\end{theor}
\begin{rem}(to be compared with the last paragraph of
section 1 of \cite[p. 271]{Getzler:BVAlg})
Since $S^3$ is a double suspension, $S^3$ is $S^1$-pointed space
(Example~\ref{suspension exemple}).
Consider the diagonal action on $\Omega^2 S^3$.
The adjunction map $X\rightarrow \Omega^2\Sigma^2 X$
is $S^1$-equivariant with respect to the trivial action on $X$
and the diagonal action on $\Omega^2\Sigma^2 X$.
So the $BV$ operator due to the diagonal action
$$
BV_{diag}:H_1(\Omega^2S^3 , \mathbb{F}_2)\rightarrow 
H_2(\Omega^2S^3 , \mathbb{F}_2)$$
is trivial.
Therefore modulo $2$, this operator $BV_{diag}$ is different from the
$BV$ operator $BV_S$ considered in Theorem~\ref{bvcalc:omegasdeux}.
On the contrary, over $\mathbb{Q}$,
$BV_{diag}$ and $BV_S$ coincide on $H_*(\Omega^2\Sigma^2 X;\mathbb{Q})$
for any connected space $X$.
Indeed,   by Theorem~\ref{Calcul BV_n algebre rationnellement}
and~(\ref{BV derivation pour le crochet}),
they both vanish on
$\pi_*(\Omega^2\Sigma^2 X)\otimes\mathbb{Q}$ which is the free Lie algebra
on $H_{>0}(X)$.
\end{rem}

\bibliography{Bibliographie_Gerald_Luc}
\bibliographystyle{amsplain}
\end{document}